\documentclass[11pt,twoside,a4paper]{article}

\usepackage{amsfonts}
\usepackage{amsmath}
\usepackage{amsthm}
\usepackage{graphicx}

\usepackage[left=3cm,right=3cm,top=2.7cm,bottom=3cm]{geometry}	
\usepackage{authblk}

\usepackage{enumerate}
\usepackage{cite}

\def\rint{{\rm int\,}}
\def\rcl{{\rm cl\,}}

\def\9{\"u9}
\newtheoremstyle{custom}{3pt}{3pt}{}{}{\bfseries}{:}{.5em}{}
\theoremstyle{custom}
\newtheorem{theorem}{Theorem}[section]

\newtheorem{definition}[theorem]{Definition}
\newtheorem{remark}[theorem]{Remark}

\newtheorem{proposition}[theorem]{Proposition}
\newtheorem{example}[theorem]{Example}

\numberwithin{equation}{section}
\numberwithin{figure}{section}
\numberwithin{table}{section}
\newenvironment{keywords}{{\bf Keywords:}}{}

\def\R{\mathbb{R}}
\def\Rp{\R_{\geq0}}

\def\N{\mathbb{N}}

\def\X{\mathbb{X}}
\def\Y{\mathbb{Y}}
\def\U{\mathbb{U}}

\def\NN{{\cal N}}

\def\eps{\varepsilon}

\newcommand{\norm}[1]{\left\lVert#1\right\rVert}

\def\beq{\begin{equation}}
	\def\eeq{\end{equation}}
\def\bea{\begin{eqnarray}}
	\def\eea{\end{eqnarray}}
\def\beaa{\begin{eqnarray*}}
	\def\eeaa{\end{eqnarray*}}

\def\ub{\mathbf{u}}
\def\tell{\tilde\ell}
\def\xk{(x_{\ub}(k,x_0),u(k))}
\newcommand{\equ}{(x^{e},u^{e})}
\newcommand{\equi}{(x_i^{e},u_i^{e})}
\newcommand{\eq}{x^{e}}
\newcommand{\equm}{(x^{e}_\mu,u^{e}_\mu)}
\newcommand{\eqm}{x^{e}_\mu}

\def\mus{\mu^\star}
\def\LL{{\cal L}}
\def\KK{{\cal K}}

\def\BB{{\cal B}}

\usepackage{xcolor}
\definecolor{darkgreen}{rgb}{0,0.55,0}

\begin{document}

	\title{Multiobjective strict dissipativity via a weighted sum approach}
		\author[1]{Lars Gr\"une}
		\author[1]{Lisa Kr\"ugel}
		\author[2]{Matthias A. M\"uller}
		
		\affil[1]{Universit\"at Bayreuth, Mathematical Institute, Chair of Applied Mathematics, Germany}
		\affil[2]{Leibniz Universit\"at Hannover, Institute of Automatic Control, Hannover, Germany}

		\date{\today}

\maketitle
\renewcommand{\thefootnote}{\arabic{footnote}}
\pagestyle{myheadings}
\thispagestyle{plain}
\markboth{LARS GR\"UNE, LISA KR\"UGEL, MATTHIAS A. M\"ULLER}{MULTIOBJECTIVE STRICT DISSIPATIVITY}

\begin{abstract}
			We consider nonlinear model predictive control (MPC) with multiple competing cost functions. This leads to the formulation of multiobjective optimal control problems (MO OCPs). Since the design of MPC algorithms for directly solving multiobjective problems is rather complicated, particularly if terminal conditions shall be avoided, we use an indirect approach via a weighted sum formulation for solving these MO OCPs. This way, for each set of weights we obtain an optimal control problem with a  single objective. In economic MPC it is known that strict dissipativity is the key assumption for concluding performance and stability results. We thus investigate under which conditions a convex combination of strictly dissipative stage costs again yields a stage cost for which the system is strictly dissipative. We first give conditions for problems with linear dynamics and then move on to consider fully nonlinear optimal control problems. We derive both necessary and sufficient conditions on the individual cost functions and on the weights to conclude strict dissipativity and illustrate our findings with numerical examples.
\end{abstract}
	

	\section{Introduction}
	%
	
	Model predictive control (MPC) is a common and popular method to solve optimal control problems on finite and infinite horizons. In MPC, at each sampling instant a finite-horizon optimal control problem is solved and the first part of the optimal input sequence is applied to the system, before the procedure is repeated at the next sampling instant. Using such a receding horizon approach, one can generate closed-loop trajectories that are approximately optimal on infinite time horizons, see, e.g., \cite{GruP17,FaGM18}. Together with the fact that hard input and state constraints can be directly specified in the controller design, this makes MPC a very attractive control method that has found widespread application in many industrial fields, see, e.g., \cite{Rakovic2019} and the references therein. While in classical stabilizing MPC, cost functions are used that are positive definite with respect to some desired setpoint to be stabilized (or, more general, trajectory to be tracked), more recent economic MPC approaches consider more general control objectives that might be related to some economic performance criteria and that do not need to satisfy such a definiteness condition \cite{GruP17,FaGM18}. 
	
	Most of the existing MPC literature considers the case where one given cost function shall be minimised. However, in many practical applications as in \cite{Kajgaard2013, Logist2010, Schmitt2020, Sauerteig2019,Flasskamp2020}, it is desirable to consider not only a single but several cost criteria. This leads to the formulation of a multiobjective optimal control problem (MO OCP). In recent research, different approaches to solve MO OCPs with MPC algorithms were given, see \cite{Gruene2019, Zavala2015,Zavala2012, Stieler2018}. They all have in common that it is not easy to find an appropriate way to formulate a multiobjective MPC (MO MPC) scheme for MO OCPs without terminal conditions, see for instance \cite[Chapter 5]{Stieler2018}. A possibility to avoid these problems is to use the weighted sum approach, see \cite{Ehrgott2005}, as done for example in \cite{Sauerteig2019}. In this approach, the individual cost functions $\ell_1,\ldots,\ell_m$ are combined into one function via the sum 
	\begin{equation} \ell = \sum_{i=1}^m \mu_i \ell_i\label{eq:convcost}\end{equation}
	with weights $\mu_i\in[0,1]$ and $\sum_i \mu_i =1$. Thus, we can use the theory and properties of MPC based on single criterion optimal control problems without terminal conditions, see \cite{GruP17,Gruene2013}. 
	
	In economic MPC, it is known that strict dissipativity is the key ingredient for asymptotic stability as well as for averaged and transient (approximate) optimality of the closed loop, see, e.g., \cite{FaGM18} or  \cite[Section 7]{GruP17}.
	Motivated by MO MPC applications, the theoretical question we thus study in this paper is: if the system is strictly dissipative for the $m\ge 2$ stage costs $\ell_i$, $i=1,2\ldots,m$, is it also strictly dissipative for the convex weighted sum 
	\[\ell_\mu = \sum_{i=1}^{m} \mu_i\ell_i, \quad \mbox{ where } \mu_i\in [0,1] \mbox{ and } \sum_{i=1}^{m} \mu_i = 1\]
	of these costs?	
	
	This question is interesting because optimal control and MPC with the cost $\ell_\mu$ and varying $\mu$ yields a possibility to compute efficient solutions (also known as Pareto-optimal solutions) to the multiobjective optimal control problem. While the weighted-sum approach does not parametrise all efficient solutions, it parametrises many of them and yields a particularly simple approach to multiobjective optimisation, see for instance \cite[Chapter 3]{Ehrgott2005}. We illustrate this well-known fact in Example \ref{ex:linearfail}, below.
	
	Summarizing, our findings provide a basis for a better understanding of the behaviour of economic MO OCPs. With the assumptions derived in this paper we can conclude strict dissipativity for convexly combined stage costs, which in turn allows to conclude approximate optimality of the corresponding MO MPC scheme. In order to keep the exposition focused and self-contained, we refrain from discussing concrete MPC schemes or performing MPC simulations in this paper. For these we refer to, e.g., \cite{FaGM18} and the references therein.
	
	In order to render our presentation in this paper less technical, we have decided to write all results and proofs for only two cost functions $\ell_1$ and $\ell_2$, in which case the weighted sum specializes to
	\[ \ell = \mu\ell_1 + (1-\mu)\ell_2, \quad \mu\in[0,1]. \]
	After each main result we provide a remark that explains how the assumptions and assertions extend to the general setting \eqref{eq:convcost}. 
	
	The paper is organised as follows: In Section \ref{sec:setting} we introduce the problem class that we are considering along with basic definitions of dissipativity and convexity as well as properties of the Karush-Kuhn-Tucker (KKT) and Lagrange multiplier theory. In Section \ref{sec: linear} we start our analysis with linear system dynamics and shed light on the question of when the convex combination of stage costs for which the system is strictly dissipative yields a stage cost for which the system is again strictly dissipative. We move on, in Section \ref{sec:nonlin}, giving a necessary condition for the continuous dependence of the optimal equilibrium on $\mu$ in case strict dissipativity holds and a sufficient condition under which strict dissipativity is preserved for small changes in $\mu$. Further, we provide two sufficient conditions which ensure local and global strict dissipativity for all $\mu\in[0,1]$. We illustrate our theoretical findings by numerical examples in each section. Section \ref{sec:con} concludes this paper.
	
	\section*{Notation}
	In the following we denote the natural numbers, the real numbers, and the non-negative real numbers by $\N$, $\R$, and $\R_{\geq0}$, respectively. The symbol $\R^n$ with $n\in\N$ denotes the $n$-dimensional Euclidean space. Moreover, with $\BB_\eps(x_0)\subseteq \R^n$ we denote the open ball with radius $\eps>0$ around $x_0$. With $\rint \Y$ and $\rcl \Y$ we denote the interior and the closure of a set $\Y\subset\R^n$, respectively. Further, $\norm{\cdot}$ is an arbitrary norm in $\R^n$ and $\norm{(\cdot,\cdot)}$ denotes a norm where two vectors are composed.

	\section{Setting and preliminaries}\label{sec:setting}
	We consider discrete time nonlinear systems of the form
	\begin{equation}\label{eq: system}
		x(k+1) = f(x(k),u(k)),\quad x(0)=x_0
	\end{equation}
	with $f: \R^n\times \R^m\to \R^n$ continuous. We denote the solution of system \eqref{eq: system} for a control sequence $\ub = (u(0),\dots, u(N-1))\in (\R^m)^{N}$ and initial value $x_0\in\R^n$ by $x_{\ub}(\cdot, x_0)$, or briefly by $x(\cdot)$ if there is no ambiguity about the respective control sequence and the initial value. Conceptually, the analysis in this paper should also be feasible in continuous time as in  \cite{Faulwasser2017}, however, additional regularity assumptions such as smoothness of the storage function may be required, which we avoid in the discrete-time setting of this paper.
	
	We impose a non-empty combined state and input constraint set $\Y\subseteq \R^n\times\R^m$ and we define the induced state and input constraint sets $\X := \{x\in \R^n\,|\,\text{ there is } u\in \R^m \text{ with } (x,u)\in \Y\}$ and $\U := \{u\in \R^m\,|\,\text{ there is } x\in \R^n \text{ with } (x,u)\in \Y\}$. 
	For some results we need that $\Y$ is of the form
	\beq \Y = \{(x,u)\in\R^n\times\R^m\,|\, g(x,u) \le 0\} \label{eq:Yg}\eeq
	for a function $g:\R^n\times\R^m\to\R^p$, where ``$\le$'' in $\R^p$ is understood componentwise.
	
	For an initial value $x_0\in\X$, the set of admissible control sequences up to time $N\in\N$ is defined by 
	\begin{align*}\U^N(x_0):= \{ \ub\in \U^{N}\,|\, (x_{\ub}(k, x_0),\ub(k))\in\Y\; \\
			\forall\; k=0,\ldots,N-1, \, x_{\ub}(N, x_0)\in\X\}.\end{align*}
	Given a stage cost function $\ell:\Y\to\R$, the cost functional for MPC-horizon $N\in\N_{\ge2}$ is given by
	\[J^N(x_0,\ub)=\sum_{k=0}^{N-1}\ell\xk.\]
	With this functional, following the standard notation, see \cite{GruP17}, we can formulate the optimal control problem 
	\begin{equation}\label{eq:ocp}
		\begin{split}
			\min_{\ub\in\U^N(x_0)}J^N(x_0,\ub)&=\sum_{k=0}^{N-1}\ell\xk\\
			\mbox{s.t.} \;\; x(k+1)&=f(x(k),u(k)),\quad	k=0,\dots,N-1\\
			x(0)&=x_0,
		\end{split}
	\end{equation}
	where we minimize the cost functional over all trajectories of the system starting in $x_0$.
	
	Further, a pair $\equ\in\Y$ is called an equilibrium if $\eq=f\equ$ holds. Note that all equilibria we consider in this paper are assumed to be admissible, i.e., to lie in $\Y$. We say that an equilibrium $\equ$ is a strictly globally optimal equilibrium if 
	\[ \ell\equ< \ell (x,u) \]
	holds for all equilibria $(x,u)\in \Y$ with $x\ne \eq$. 
	
	In the following we will make use of comparison-functions defined by
\begin{align*}
	\KK :=\{\alpha:\Rp\to\Rp&\mid \alpha \text{ is continuous and}\\
	&\quad\text{strictly increasing with }\alpha(0)=0\},\\
	\KK_\infty :=\{\alpha:\Rp\to\Rp&\mid\alpha\in\mathcal{K},\  \alpha \text{ is unbounded}\},\\
	\LL:=\{\delta:\Rp\to\Rp&\mid\delta\text{ is continuous and}\\
	&\quad\text{strictly decreasing with} \lim_{t\to\infty}\delta(t)=0\}.
\end{align*}
	
	In this paper we focus on the question if and when the weighted sum of strictly dissipative stage costs is again strictly dissipative. To this end, we recall the definition of strict dissipativity, see for instance \cite{GruP17}. In addition, we make use of the notion of strict pre-dissipativity, see \cite{GruG18}, and of the notion of strict $(x,u)$-dissipativity.
	\begin{definition}[Strict (pre-)dissipativity]\label{def: strDiss}
		(i) System \eqref{eq: system} is called \emph{strictly pre-dissipative} w.r.t. the supply rate $s(x,u)$
		at an equilibrium $\equ$ if there exists a storage function $\lambda:\X \to\R$, bounded on bounded subsets of $\X$, and a function $\alpha\in\KK_\infty$, such that for all $(x,u)\in\Y$ with $f(x,u)\in \X$ the inequality
	\begin{equation}\label{eq:diss}
				s(x,u)+\lambda(x)-\lambda(f(x,u))\geq \alpha(\norm{x-\eq})
	\end{equation}
		holds. 
		
		(ii) The system 
		 is called \emph{strictly dissipative} if it is strictly pre-dissipative and the storage function $\lambda:\X\to\R$ is bounded from below on $\X$. 
		
		(iii) The system
		is called \emph{strictly $(x,u)$-dissipative} if the same holds with the inequality
			\begin{align*}
				s(x,u)+\lambda(x)-\lambda(f(x,u))\geq \alpha(\norm{x-\eq,u-u^{e}}).
		\end{align*}
		
		(iv) The system 
		is called \emph{locally strictly (pre/$(x,u)$)-dissipative}, if there exists a neighbourhood $\NN$ of $(x^e,u^e)$ such that (i), (ii), or (iii), respectively, hold with $\Y=\rcl \NN$.
	
	(v) If the supply rate is of the form $s(x,u)=\ell(x,u)-\ell\equ$, with stage cost $\ell$ from the optimal control problem \eqref{eq:ocp}, then we say that the system \eqref{eq: system} is \emph{strictly dissipative (pre-dissipative, \ldots) for the stage cost $\ell$ and} call
		\begin{equation}
			\tell(x,u):=\ell(x,u)-\ell\equ +\lambda(x)-\lambda(f(x,u))
		\end{equation}
	the \emph{rotated stage cost}.
\end{definition}

	\begin{remark}
		The notion of strict pre-dissipativity we used above is quite similar to the notion of cyclo-dissipativity, see for instance \cite{Hill1980, Schaft2021}. Yet, there are also certain differences. For instance, cyclo-dissipativity allows that the storage function is unbounded on bounded sets. Further, cyclo-dissipativity requires controllability and detectability to ensure the existence of a storage function in each $x$. Since we assume the existence of such a storage function in the following, pre-dissipativity is the more suitable property for us to work with in this paper.
	\end{remark}

	The following standard definition will be used in various contexts in this paper.
	
	\begin{definition} \label{def:convex} (i) A set $D\subset\R^n$ is called {\em convex}, if for all $x_1,x_2\in D$ and $\mu\in[0,1]$ the relation $\mu x_1 + (1-\mu)x_2\in D$ holds.
		
		(ii) A scalar valued function $h:D\to \R$, $D\subset\R^n$ convex, is called {\em convex} if 
		\[ h(\mu x_1 + (1-\mu)x_2) \le \mu h(x_1) + (1-\mu)h(x_2) \]
		holds for all $\mu\in[0,1]$ and {\em strictly convex} if this inequality is strict for all $\mu\in(0,1)$. A vector valued function $h=(h_1,\ldots,h_p)^T:D\to \R^p$ is called (strictly) convex if all its component functions $h_i: D\to\R$, $i=1,\ldots,p$, are (strictly) convex.
		
		(iii) For functions $h_i:D\to\R$, $i=1, \ldots, m$, we call 
		\[ \sum_{i=1}^m \mu_i h_i \quad \mbox{for } \mu_i\in[0,1], \; \sum_{i=1}^m \mu_i=1\]
		their {\em convex combination}. For two functions $h_1,h_2:D\to\R$, this simplifies to
		\[ \mu h_1 + (1-\mu)h_2 \quad \mbox{for } \; \mu\in[0,1].\]
	\end{definition}
	It is easily seen that whenever $\ell$ is strictly convex, its minimiser is an equilibrium, and either $\Y$ is bounded or $\ell$ grows unboundedly for $\|x\|\to\infty$, then we can conclude that strict dissipativity holds even with $\lambda\equiv 0$.
	
	We now provide a couple of preliminary results on strict (pre-) dissipativity, which will be needed in the remainder of the paper. It is immediate from the definition of strict (pre-)dissipativity that $(x^e,u^e)$ is a globally optimal equilibrium w.r.t.\ the stage cost, i.e., it satisfies $\ell(x^e,u^e)\le\ell(\hat x, \hat u)$ for all other equilibria $(\hat x, \hat u)$. This means that $(x^e,u^e)$ is a minimiser of the optimisation problem
	
	\beq \label{eq:opteq}
	\begin{array}{l}
		\displaystyle \min_{(x,u)\in \Y} \ell(x, u)\\[2ex]
		\displaystyle \mbox{s.t. } \, x - f(x,u) = 0.
	\end{array}
	\eeq
	If the minimiser $(x^e,u^e)$ lies in the interior $\rint\Y$ of $\Y$, then it is also a local minimum of the problem
	\beq \label{eq:opteqmin}
	\begin{array}{l}
		\displaystyle \min_{(x,u)\in \R^n\times\R^m} \ell(x, u)\\[2ex]
		\displaystyle \mbox{s.t. } \, x - f(x,u) = 0.
	\end{array}
	\eeq
	If $\ell$ and $f$ are $C^1$, then this implies that there exists a Lagrange multiplier $\nu^e\in\R^n$ such that $(x^e,u^e)$ and $\nu^e$ satisfy the necessary optimality or KKT conditions
	\beq \label{eq:necopt}
	\begin{split}
	\dfrac{\partial L}{\partial x}(x^e,u^e,\nu^e) &= 0, \quad
	\dfrac{\partial L}{\partial u}(x^e,u^e,\nu^e) = 0, \\
	\dfrac{\partial L}{\partial \nu}(x^e,u^e,\nu^e) &= 0
	\end{split}
	\eeq
	for the Lagrange function 
	\beq \label{eq:Lagrange}
	L(x,u,\nu)=\ell(x,u)+\nu^T(x-f(x,u)).\eeq

	The next two results apply to optimal control problems with linear dynamics
	\beq 
	x^+ = Ax + Bu\label{eq:lindyn}
	\eeq
	with $A\in\R^{n\times n}$ and $B\in\R^{n\times m}$. As already stated after Definition \ref{def:convex}, strict dissipativity holds for convex stage costs if the minimiser is an equilibrium. The following two results generalise this fact. The first shows that under appropriate technical conditions the same is true also if the minimiser of $\ell$ is not necessarily an equilibrium.

	\begin{proposition} \label{prop:convl} Consider the optimal control problem \eqref{eq:ocp} with linear dynamics \eqref{eq:lindyn}, strictly convex stage cost $\ell$, and constraint set $\Y$ defined via \eqref{eq:Yg} with a convex function $g$. Assume that problem \eqref{eq:opteq} has a global minimum $(x^e,u^e)$ and satisfies the following Slater condition: There exists a pair $(\hat x, \hat u)\in\R^n\times\R^m$ with
		\[ g (\hat x, \hat u) < 0 \;\; \mbox{ and } \;\; \hat x - A \hat x - B \hat u = 0.\]
		Then, there exists a vector $\nu\in \R^n$ such that the system is strictly pre-dissipative for the stage cost $\ell$ and $\lambda(x) = \nu^T x$.
	\end{proposition}
	For a proof of this proposition see \cite[Proposition 4.3]{DGSW14}. Note that the Slater condition is satisfied with $(\hat x, \hat u)=(x^e,u^e)$ if $(x^e,u^e)\in\rint \Y$.

	The second result shows that strict dissipativity also holds if the stage cost is not itself convex, but can be appropriately bounded by a convex function.

	\begin{proposition} Consider the optimal control problem \eqref{eq:ocp} with linear dynamics \eqref{eq:lindyn} and constraint set $\Y$ defined via \eqref{eq:Yg} with a convex function $g$. Assume there is a strictly convex function $\hat \ell$ with $\hat\ell \le \ell$ and $\hat\ell\equ=\ell\equ$ for the strictly globally optimal equilibrium $\equ\in\Y$ of $\hat \ell$, and that the Slater condition from Proposition \ref{prop:convl} holds. Then the system is strictly pre-dissipative for the stage cost $\ell$ with linear storage function.
		\label{prop:convbound}\end{proposition}
	\begin{proof} From Proposition \ref{prop:convl} it follows that there exists a linear, hence continuous storage function $\lambda$ for the stage cost $\hat \ell$. For this storage function we thus obtain
		\begin{align*}
		\lambda&(f(x,u)) \\
		& \le  \lambda(x) + \hat \ell(x,u) - \hat \ell\equ - \alpha(\|x-\eq\|) \\
		& \le  \lambda(x) + \ell(x,u) - \ell\equ - \alpha(\|x-\eq\|)
		\end{align*}
		for all $(x,u)\in\Y$. This proves strict pre-dissipativity for the stage cost $\ell$.
	\end{proof}
	
	We note that both Proposition \ref{prop:convl} and Proposition \ref{prop:convbound} yield strict dissipativity if $\Y$ is compact, since then $\X$ is compact, too, and the linear storage function is bounded from below on $\X$.
	
	The final result of this section shows that if $(x^e,u^e)\in\rint\Y$, then the linear part of the storage function, i.e., its gradient, always coincides with the Lagrange multiplier $\nu$ from the necessary optimality conditions \eqref{eq:necopt}. Here no linearity assumption on $f$ is needed.
	
	\begin{proposition} \label{prop:lambdaLagrange}Consider the optimal control problem \eqref{eq:ocp} and assume that the system
		is strictly dissipative for the stage cost at an equilibrium $(x^e,u^e)\in\rint\Y$. Assume that $f$, $\ell$ and $\lambda$ are $C^1$. Then there exists a Lagrange multiplier $\nu^e$ satisfying \eqref{eq:necopt} such that the identity
		\[ \nabla_x \lambda(x^e) = \nu^e \]
		holds.
	\end{proposition}
	
	This result has first been used in the proof of \cite[Theorem 5]{Muller_et_al_TAC15} and has also been proven in \cite[Theorem 3]{Faulwasser2018}.


	\section{Results for linear dynamics}\label{sec: linear}
	
	We now turn to the investigation of strict dissipativity of optimal control problems with convexly combined stage costs. We start our consideration with the case of linear dynamics \eqref{eq:lindyn}. 
	
	\subsection{Linear dynamics with quadratic costs}
	
	In the case of linear dynamics, a common choice of stage costs are quadratic costs 
	\beq
	\ell_i(x,u) = x^TQ_i x+u^TR_i u + s_i^Tx + v_i^Tu,\quad i=1,2\label{eq:quadcost}
	\eeq
	with $Q_i\in\R^{n\times n}$, $R_i\in\R^{m\times m}$, $s_i\in\R^n$ and $v_i\in\R^m$. Here we assume that $Q_i$ and $R_i$ are symmetric, $Q_i$ is positive semidefinite and $R_i$ is positive definite. These stage costs are also called generalised quadratic costs, as they also contain linear terms. In this setting, dissipativity characterisations and storage functions can be explicitly computed using the techniques from \cite{GruG18}. This leads to the following theorem.
	%
	%
	
	\begin{theorem} \label{thm:lq} Consider the optimal control problem \eqref{eq:ocp} with linear dynamics \eqref{eq:lindyn} and quadratic costs $\ell_1$ and $\ell_2$ \eqref{eq:quadcost}. We assume that the constraint set $\Y$ is either convex and compact or $\Y=\R^n\times\R^m$, i.e., there are no constraints, and that the system \eqref{eq: system} 
		is strictly dissipative for both $\ell_1$ and $\ell_2$. In case $\Y$ is compact, we assume that the global optimal equilibria $(x_\mu^e, u_\mu^e)$ minimising \eqref{eq:opteq} for $\ell = \ell_\mu := \mu \ell_1 + (1-\mu)\ell_2$ satisfy $(x_\mu^e,u_\mu^e)\in\rint\Y$ for all $\mu\in[0,1]$. 
		Then the system 
		is strictly dissipative at $\equm$ for cost function $\ell_\mu$ for all $\mu\in[0,1]$.
	\end{theorem}
	
	\begin{proof} 
		By \cite[Lemma 4.1]{GruG18}, since $(x_\mu^e,u_\mu^e)\in\rint\Y$ for $\mu=0$ and $\mu=1$, strict pre-dissipativity for $\ell_i$, $i=1,2$ holds if and only if there is a symmetric solution $P_i$ of the matrix inequality
		\beq 
		Q_i + P_i - A^T P_i A > 0, \label{eq:lmi}
		\eeq
		with $Q_i$ from \eqref{eq:quadcost} being symmetric and positive semidefinite.

		In this case, the storage function can be chosen to be of the linear-quadratic form
		\[ \lambda_i (x) = x^TP_ix + p_i^Tx\]
		for an appropriate vector $p_i\in\R^n$ (see the discussion after this proof). In case $\Y=\R^n\times\R^m$ the matrix $P_i$ must be positive semidefinite for $\lambda_i$ to be bounded from below. Then, however, we may choose $P_i$ to be positive definite, because when a positive semidefinite matrix $P_i$ satisfies \eqref{eq:lmi}, then for sufficiently small $\eps>0$ the positive definite matrix $P_i + \eps I$ also satisfies \eqref{eq:lmi}.
		
		Thus, the assumptions of the theorem imply that we can find symmetric solutions $P_1$, $P_2$ of \eqref{eq:lmi} for $i=1,2$, respectively, which we can choose to be positive definite if $\Y=\R^n\times\R^m$. It is then easily seen that $P_\mu = \mu P_1 + (1-\mu)P_2$ solves \eqref{eq:lmi} for $Q_\mu = \mu Q_1 + (1-\mu)Q_2$. Hence, since $(x_\mu^e,u_\mu^e)\in\rint\Y$, by \cite[Lemma 4.1]{GruG18} strict pre-dissipativity for the cost $\ell_\mu$ follows.
		
		If $\Y=\R^n\times\R^m$, $P_\mu$ is positive definite and thus boundedness from below of $\lambda_\mu(x) = x^TP_\mu x + p_\mu^Tx$ holds (regardless of what $p_\mu$ is). If $\Y$ is compact, $\lambda_\mu$ is bounded from below on $\X$ by its continuity. Thus, in both cases $\lambda_\mu$ is bounded from below on $\X$, which together with strict pre-dissipativity implies strict dissipativity.
	\end{proof}
	
	\begin{remark} 
		\begin{enumerate}
			\item Theorem \ref{thm:lq} and its proof immediately generalise to $m$ cost functions $\ell_1,\ldots,\ell_m$ of the form \eqref{eq:quadcost}, since it is easy to check that $P_\mu = \sum_i\mu_i P_i$ solves \eqref{eq:lmi} for $Q_\mu = \sum_i \mu_i Q_i$ if each $P_i$ for $i=1,\ldots,m$ solves \eqref{eq:lmi} for $Q_i$.
			\item In this generalised proof of Theorem \ref{thm:lq}, the requirement $\sum_i \mu_i=1$ is not necessary. Assuming the weights to satisfy $\mu_i\geq0$ for all $i$ and $\mu_i>0$ for at least one $i$ would suffice. The same is true for the $m$-cost variants of all subsequent results in this paper. Yet, since this paper is motivated by the weighted sum approach, we formulate all our results with convex combinations, i.e. with the assumption that $\sum_i\mu_i=1$.
		\end{enumerate}
	\end{remark}
	
	The construction in the proof implies that the matrix $P_\mu$ in the quadratic storage function $\lambda_\mu$ for $\ell_\mu$ is a weighted sum $P_\mu = \mu P_1 + (1-\mu)P_2$ of the matrices $P_1$ and $P_2$ in the storage functions for $\ell_1$ and $\ell_2$. Moreover, it follows from Proposition \ref{prop:lambdaLagrange} that the vector $p_\mu$ is the Lagrange multiplier $\nu_\mu=\nu$ of the optimisation problem \eqref{eq:opteq} with $\ell = \ell_\mu$. 
	In contrast to the matrix $P_\mu$, this vector $p_\mu$ can in general {\em not} be chosen as a convex combination $p_\mu = \mu p_1 + (1-\mu)p_2$. Likewise, the optimal equilibrium $(x_\mu^e,u_\mu^e)$ is in general not a weighted sum of $(x_1^e,u_1^e)$ and $(x_2^e,u_2^e)$. %
	%
	The following one-dimensional example illustrates this fact.
	
	\begin{example}\label{ex: linquad}
		Consider the one-dimensional dynamics $x^+ = ax+bu$ with $\Y=\R\times\R$ and the costs of the form
		\[
		\ell_i(x,u) = q_ix^2+r_i u^2+s_ix+v_iu
		\]
		with $q_i,r_i,s_i,v_i>0$, $i\in\{1,2\}$, and $a,b\in\R$. Since $q_i>0$, the stage costs $\ell_i$ are both strictly convex function. Moreover, $P_i=0$ solves the one-dimensional matrix inequality~\eqref{eq:lmi} for $Q_i=q_i$ and $A=a$, thus we do not need a quadratic part in the storage function. We do, however, in general need a linear part, which we compute as described above via the Lagrange multiplier of the optimisation problem
		\begin{align*}
			& \min_{(x,u)\in\Y} \;\ell_i(x,u) = q_ix^2+r_i u^2+s_ix+v_iu\\
			& \text{s.t. } x =ax+bu.
		\end{align*}
		For this purpose we define the Lagrange function
		\[L_i(x,u,p_i)=\ell_i(x,u)+\nu_i(x-ax-bu)\]
		and calculate the corresponding derivatives
		\begin{align}
			\dfrac{\partial L_i}{\partial x}(x,u,\nu_i)&= 2q_i x+ s_i+\nu_i (1-a)  \label{eq: dx}\\
			\dfrac{\partial L_i}{\partial u}(x,u,\nu_i)&= 2r_i u + v_i - \nu_i b  \label{eq: du}\\
			\dfrac{\partial L_i}{\partial p_i}(x,u,\nu_i)&= x -ax-bu. \label{eq: dmu}
		\end{align}
		Solving the system \eqref{eq: dx}--\eqref{eq: dmu} yields
		\begin{equation}\label{equ i}
			\begin{split}
				x_i^{e}&=\dfrac{(1-a+b)b(-bs_i-(1-a)v_i)}{(1-a+b)(2q_i b^2+2(1-a)^2r_i)}=\dfrac{b(-bs_i-(1-a)v_i)}{2(q_i b^2+(1-a)^2r_i)},\\
				u_i^{e}&= \dfrac{(1-a)(-bs_i-(1-a)v_i)}{2(q_ib^2+(1-a)^2r_i)}\\
				\nu_i&=\dfrac{1}{1-a+b}\left(\dfrac{(-bs_i-(1-a)v_i)(r_i(1-a)-bq_i)}{q_ib^2+(1-a)^2r_i}+v_i-s_i\right)\\
				&=\dfrac{q_i v_i b -(1-a)r_i s_i}{q_i b^2+(1-a)^2r_i}.
			\end{split}
		\end{equation}
		According to Proposition \ref{prop:convl}, the system 
		is strictly dissipative for both scalar stage costs at $\equi$ with storage function $\lambda_i(x)=\nu_i x$.
		
		Next, we consider the combined stage cost
		\begin{equation}
			\ell_\mu(x,u)=\mu \ell_1(x,u)+(1-\mu)\ell_2(x,u),
		\end{equation}
		which is again strictly convex for all $\mu\in[0,1]$, since a convex combination of strictly convex functions is strictly convex. As above, we define the Lagrange function
		\[L(x,u,\nu_\mu)=\ell_\mu(x,u)+\nu_\mu(x-ax-bu),\]
		and solve
		\begin{alignat}{2}
			\dfrac{\partial L}{\partial x}(x,u,p_\mu)&=\mu(2q_1x+s_1)+(1-\mu)(2q_2x+s_2)+\nu_\mu(1-a)&&=0\\
			\dfrac{\partial L}{\partial u}(x,u,p_\mu)&=\mu(2r_1 u+v_1)+(1-\mu)(2r_2 u+v_2)-\nu_\mu b&&=0\\
			\dfrac{\partial L}{\partial \nu_\mu}(x,u,p_\mu)&=x-ax-bu&&=0.
		\end{alignat}
		The solution is given by
		\begin{align}
			x^{e}_\mu &=\dfrac{-b(\mu(b s_1+(1-a)v_1)+(1-\mu)(bs_2+(1-a)v_2))}{2(\mu(b^2q_1+(1-a)^2r_1)+(1-\mu)(b^2q_2+(1-a)^2r_2))}\\
			u^{e}_\mu &=\dfrac{-\left((1-a)b(\mu s_1+(1-\mu)s_2)+(1-a)^2(\mu v_1+(1-\mu)v_2)\right)}{2\left(\mu(b^2q_1+(1-a)^2r_1)+(1-\mu)(b^2q_2+(1-a)^2r_2)\right)}
		\end{align}
	\begin{footnotesize}
	\begin{align}
		\nu_\mu	&=
				 \dfrac{b(\mu q_1+(1-\mu)q_2)(\mu v_1+(1-\mu)v_2)-(1-a)(\mu r_1+(1-\mu)r_2)(\mu s_1+(1-\mu)s_2)}{b^2(\mu q_1+(1-\mu)q_2)+(1-a)^2(\mu r_1+(1-\mu)r_2)}.
		\end{align}
	\end{footnotesize}
		
		A simple numerical example using the system $x^+ = 2x+4u$ with stage costs
		\beq 
		\ell_1(x,u) = 0.1x^2+10u^2+6x+7u \quad \mbox{ and } \quad 
		\ell_2(x,u)=4x^2+3u^2+3x+8u
		\label{eq:numcost}\eeq
		shows that $\nu_\mu$ does indeed not depend linearly on $\mu$, cf.\ Figure \ref{fig:pmu1}.
		
		\begin{figure}[h]
			\begin{center}
				\includegraphics[width=0.5\textwidth]{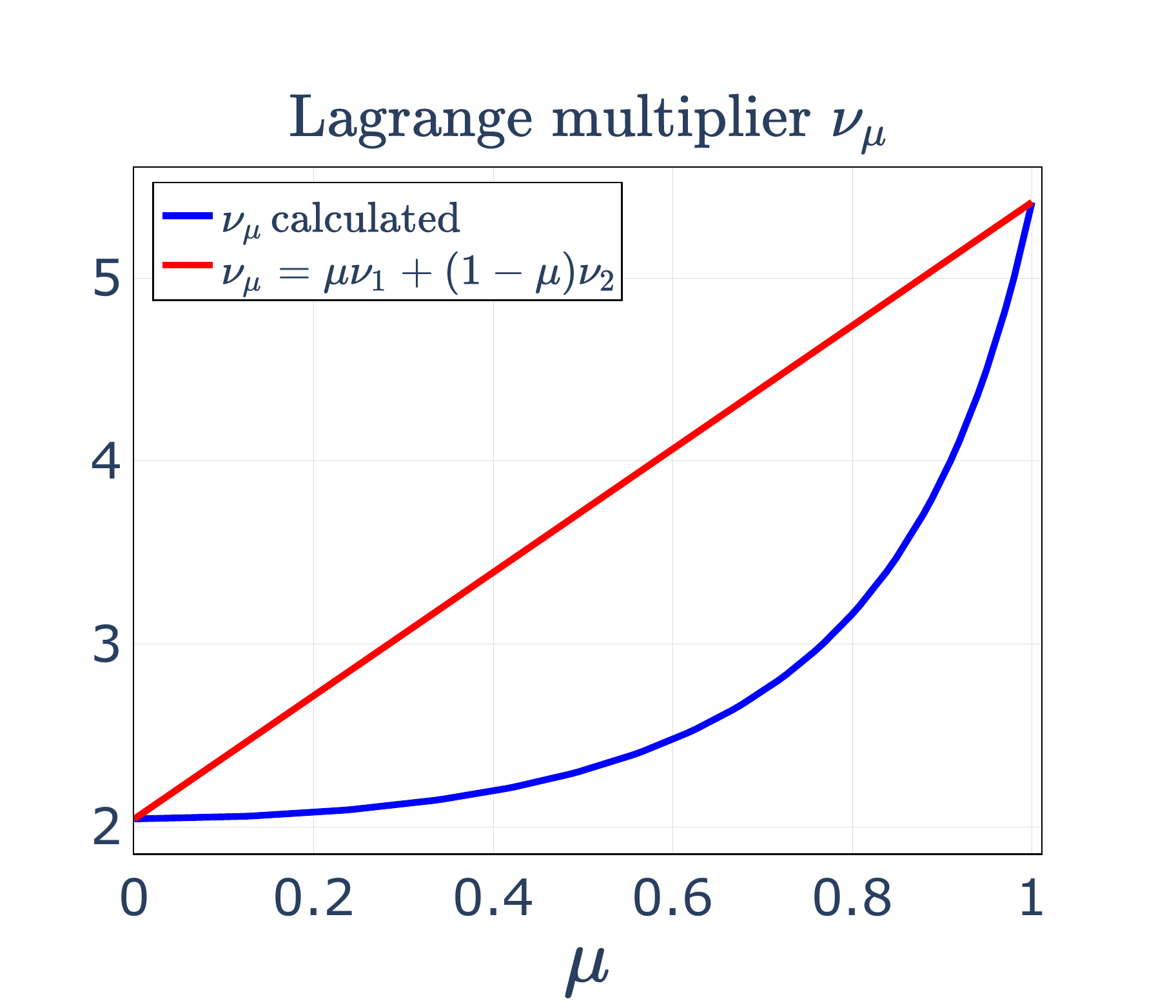}
				\caption{Lagrange multiplier $\nu_\mu$ depending on $\mu$ (blue) and convex combination $\mu \nu_1 + (1-\mu)\nu_2$ (red) for costs from \eqref{eq:numcost} \label{fig:pmu1}}
			\end{center}
		\end{figure}
		
		In order to find out whether there are cases in which $\nu_\mu$ depends linearly on $\mu$ we consider
		\begin{equation}\label{eq:pmu}
			\nu_\mu - \mu \nu_1-(1-\mu)\nu_2= 0.
		\end{equation}
		It turns out---by straightforward calculations---that sufficient conditions for this equation to hold for all $\mu\in[0,1]$ are $a=1$ or $q_1r_2 = q_2r_1$. This illustrates that only in very special cases equality~\eqref{eq:pmu} holds for arbitrary $\mu\in(0,1)$.
		\end{example}

		\subsection{Linear dynamics with non-quadratic cost}
		
		In this subsection we continue considering linear dynamics \eqref{eq:lindyn} but now with not necessarily quadratic costs. Instead, we assume that the occurring cost functions are convex or strictly convex. Then we can make use of  Proposition \ref{prop:convl} to show strict dissipativity for the combined stage costs. 
		
		\begin{theorem} \label{thm:convex} Consider the optimal control problem \eqref{eq:ocp} with linear dynamics \eqref{eq:lindyn} and convex stage costs $\ell_1$ and $\ell_2$, where at least one of the functions $\ell_1$ and $\ell_2$ is also strictly convex. Assume that the constraint set $\Y$ is convex and compact and that the Slater condition from Proposition \ref{prop:convl} is satisfied. Then the system 
			is strictly dissipative at the optimal equilibrium $\equm$ for the cost function $\ell_\mu = \mu \ell_1 + (1-\mu)\ell_2$ for all $\mu\in(0,1)$.
		\end{theorem}
		\begin{proof} One easily checks that strict convexity of either $\ell_1$ or $\ell_2$ together with convexity of the other stage cost implies strict convexity of $\ell_\mu = \mu \ell_1 + (1-\mu)\ell_2$ for all $\mu\in(0,1)$. With this observation, the claim follows from  Proposition \ref{prop:convl}. 
		\end{proof}
		
		\begin{remark}
			Obviously, when the assumptions of the theorem hold and both $\ell_1$ and $\ell_2$ are strictly convex, then we obtain strict dissipativity for all $\mu$ in the closed interval $[0,1]$. Theorem \ref{thm:convex} generalises easily to $m$ convex cost functions of which at least one is strictly convex, provided at least one $\mu_i$ corresponding to a strictly convex $\ell_i$ is not zero.
		\end{remark}

		\section{Results for nonlinear dynamics}\label{sec:nonlin}
		
		In this section we turn to the analysis of problems with general nonlinear dynamics \eqref{eq: system} and not necessarily convex stage costs. We start by presenting a necessary condition and continue with the derivation of several sufficient conditions for strict dissipativity of $\ell_\mu$ for $\mu\in(0,1)$.
		
		\subsection{A necessary condition}
		
		The following theorem identifies a property of the optimal equilibrium in case strict dissipativity holds. Since strict disspativity of the convexly combined stage cost may not hold for all weights $\mu\in[0,1]$ we state the next theorem for arbitrary compact subintervals.
		
		\begin{theorem}
			Assume that the system \eqref{eq: system} 
			is strictly dissipative for the cost function $\ell_\mu = \mu \ell_1 + (1-\mu)\ell_2$ for all $\mu\in[\underline{\mu},\overline{\mu}]\subseteq[0,1]$ and assume that the corresponding optimal equilibria $\equm$ are contained in a compact set $\widehat \Y \subset \Y$.  Then the map 
			\[ \mu \mapsto \eqm \]
			is continuous on $[\underline{\mu},\overline{\mu}]$. 
			\label{thm:discontinuous}\end{theorem}
		\begin{proof} 
			It follows directly from strict dissipativity that 
			\beq \ell_\mu\equm <  \ell_\mu(x,u) \label{eq:strictmin}\eeq
			for all equilibria $(x,u)\in\Y$ with $x\ne \eqm$. Now assume that $\mu \mapsto \eqm$ is discontinuous at some $\mus\in[\underline{\mu},\overline{\mu}]$. Then, there is a sequence $\mu_n\to \mus$ in $[\underline{\mu},\overline{\mu}]$, such that $(x^e_{\mu_n},u^e_{\mu_n})$ converges to $(\hat x^e_{\mus},\hat u^e_{\mus})$ with $\hat x^{e}_{\mus} \ne \eq_{\mus}$. Since $f(\eq_{\mu_n},u^e_{\mu_n}) = \eq_{\mu_n}$, by continuity of $f$ we have that $f(\hat x^{e}_{\mus},\hat u^e_{\mus})=\hat x^{e}_{\mus}$, i.e., the limit is an equilibrium. Using continuity of $\ell_\mu$ and inequality \eqref{eq:strictmin} for $\mu=\mu_n$ and $(x,u)=(x^e_{\mus},u^e_{\mus})$, for this equilibrium it holds that
			\[ \ell_{\mus}(\hat \eq_{\mus},\hat u^e_{\mus}) = \lim_{n\to\infty} \ell_{\mu_n}(\eq_{\mu_n},u^e_{\mu_n}) \le  \lim_{n\to\infty} \ell_{\mu_n}(\eq_{\mus},u^e_{\mus})  =  \ell_{\mus}(\eq_{\mus},u^e_{\mus}).\]
			This, however, means that inequality \eqref{eq:strictmin} does not hold at $\mu=\mus$, which yields a contradiction.
		\end{proof}
		
		\begin{remark} The reasoning in the proof immediately carries over to higher dimensional $\mu = (\mu_1,\ldots,\mu_m)^T$. Thus, if the system is strictly dissipative for stage cost $\ell_\mu = \sum_{i}\mu_i\ell_i$ for all $\mu$ from a subset $\Omega\subset \{\mu\in[0,1]^m\,|\, \sum_i \mu_i = 1\}$, then $\mu \mapsto \eqm$ is continuous on $\Omega$.
		\end{remark}
		
		The theorem in particularly implies that if the globally optimal equilibria $\eqm$ change discontinuously with $\mu$, then strict dissipativity cannot hold. While the theorem is valid for general nonlinear dynamics, discontinuity of $\eqm$ can occur even in the case of linear dynamics, as the following example shows.
		
		\begin{example}
			Consider the dynamics $x^+ = x+u$ and the cost functions
			\[ \ell_1(x,u) = \frac12 x^4 - \frac14 x^3 - x^2 + \frac34 x  + u^2 \quad \mbox{ and } \quad 
			\ell_2(x,u) = (x-1)^2  + u^2, \]
			see Figure \ref{fig:costs}, with compact constraint set $\Y=[-10,10]^2$. 
			
			\begin{figure}[h]
				\begin{center}
					\includegraphics[width = 0.4\textwidth]{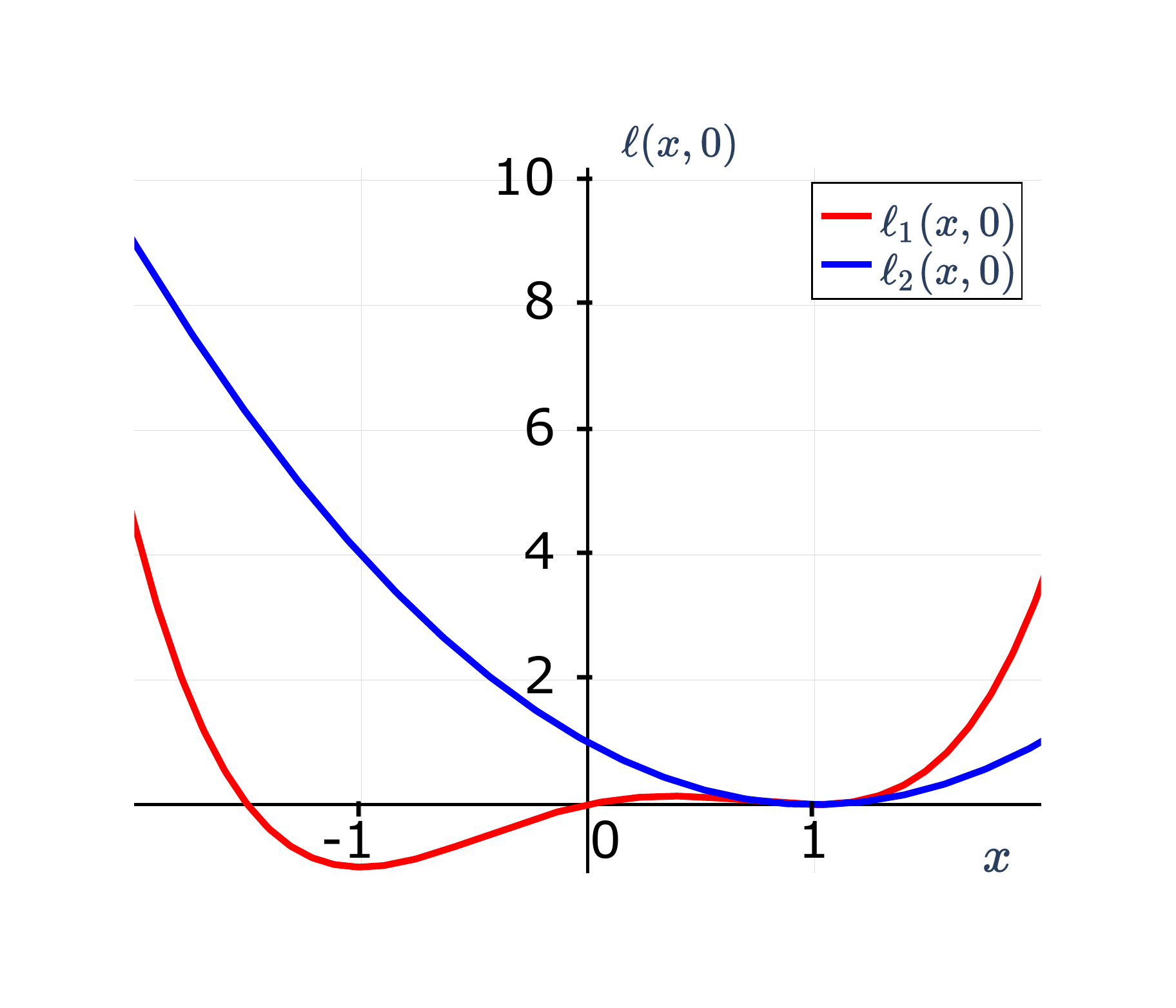}
					\caption{Graphs of cost functions $x\mapsto \ell_1(x,0)$ (red) and $x\mapsto \ell_2(x,0)$} (blue)\label{fig:costs}
				\end{center}
			\end{figure}
			
			The cost $\ell_1$ can be bounded from below by the convex cost $\hat\ell_1(x) = (x+1)^2/10-1+u^2$ and the global minimum $(-1,0)$ of $\ell_1$ and $\hat \ell_1$ coincides and is an equilibrium of the dynamics. Hence, the system 
			is strictly dissipative for stage cost $\ell_1$ by Proposition \ref{prop:convbound}. Since the cost $\ell_2$ is convex, the system 
			is strictly dissipative for stage cost $\ell_2$ by  Proposition \ref{prop:convl}. Since every $x\in\R$ is an equilibrium of the dynamics and the corresponding equilibrium control $u=0$ minimises $\ell_i$ with respect to $u$, the globally optimal equilibrium $(x^e_\mu,u^e_\mu)$ for cost $\ell_\mu=\mu\ell_1 + (1-\mu)\ell_2 +u^2$ coincides with the global minimum of $\ell_\mu$. This global minimum, however, changes discontinuously at $\mu^*=32/41$, as Figure \ref{fig:minima} shows.
			
			\begin{figure}[h]
				\begin{center}
					\includegraphics[width = 0.32\textwidth]{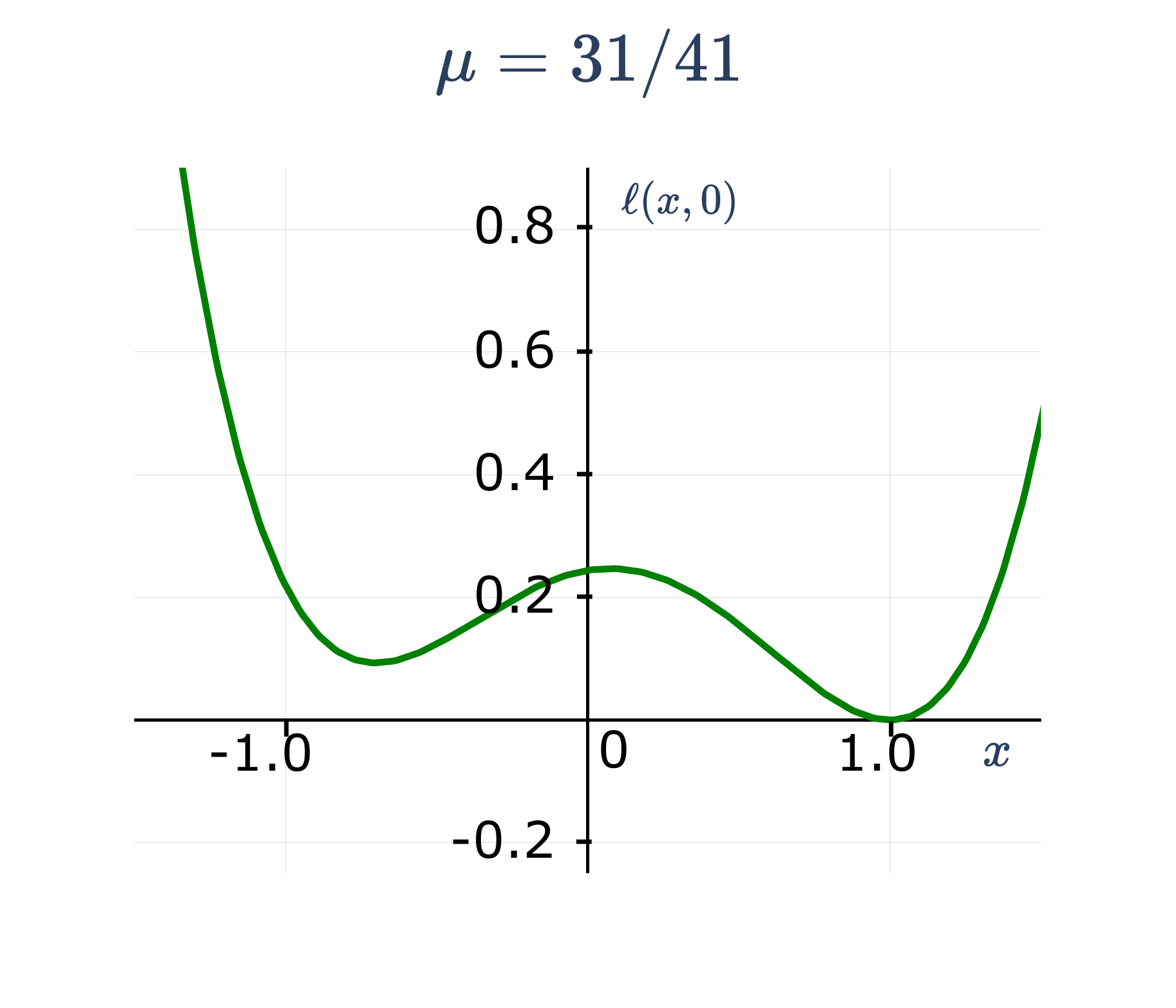}
					\includegraphics[width = 0.32\textwidth]{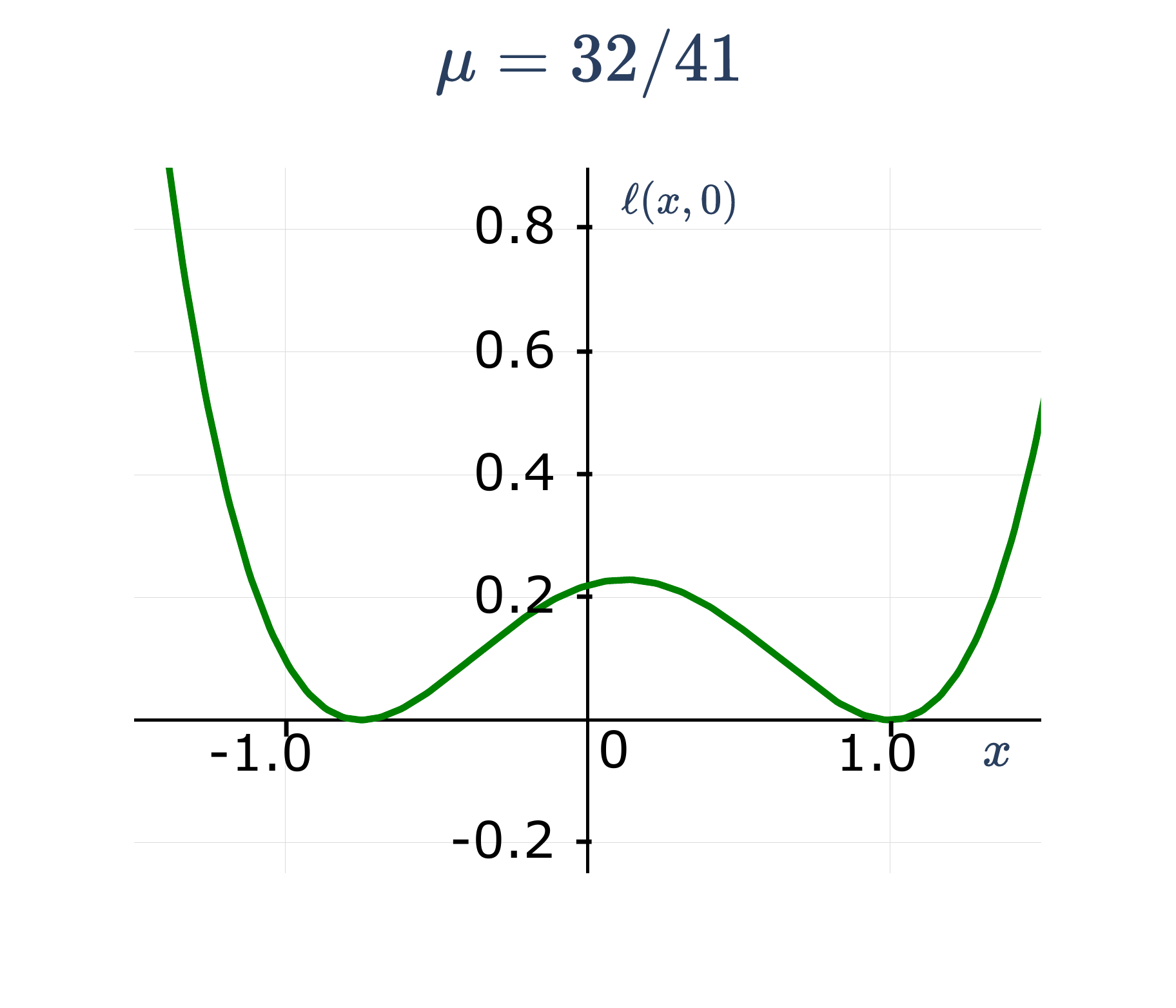}
					\includegraphics[width = 0.32\textwidth]{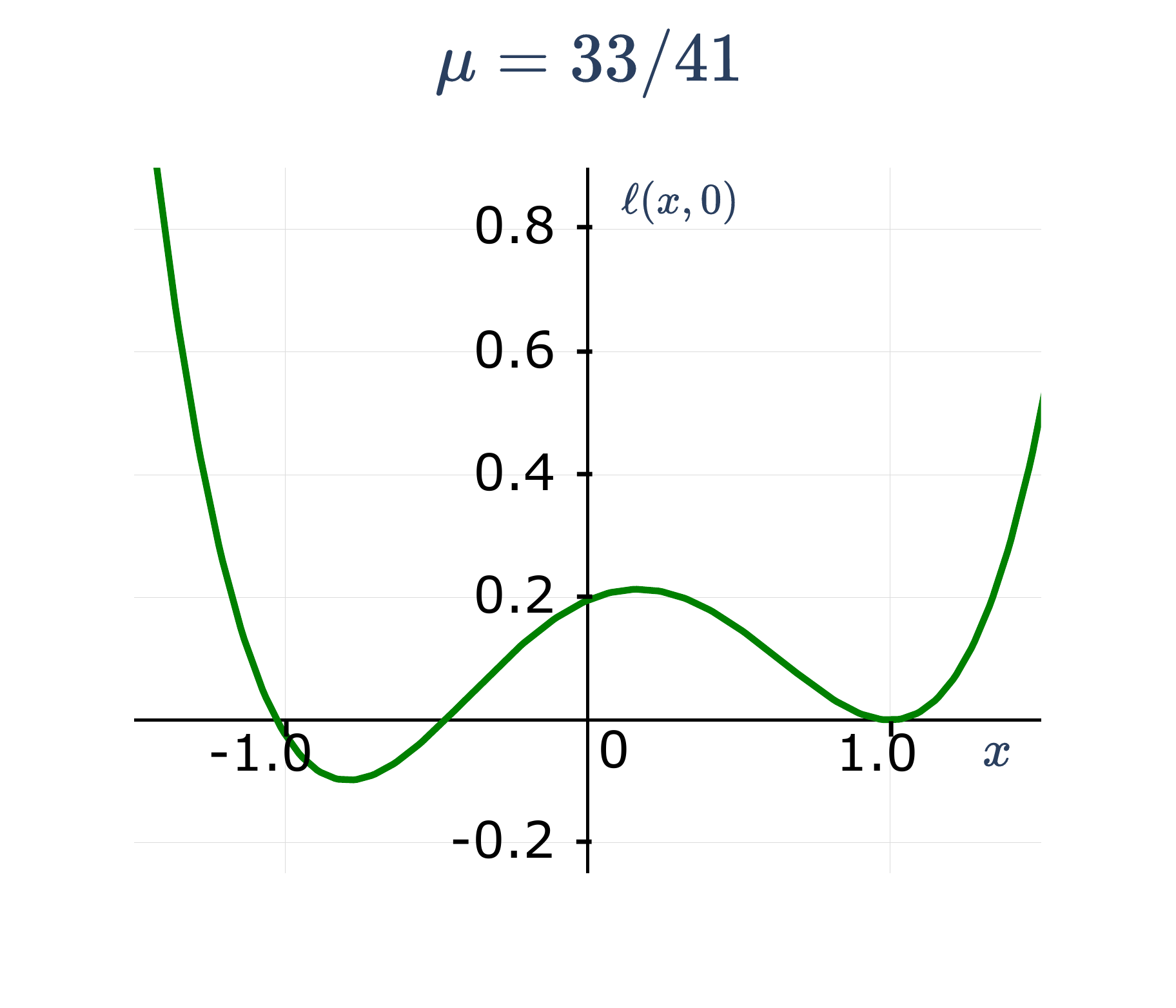}
					\caption{Graphs of cost functions $x\mapsto \ell_\mu(x,0)$ for $\mu=33/41$, $32/41$, $31/41$ (left to right)\label{fig:minima}}
				\end{center}
			\end{figure}
			
			We note that choosing the cost $\ell_1$ non-convex is crucial in this example, since for linear dynamics, convex cost $\ell_1$ and strictly convex cost $\ell_2$, Theorem \ref{thm:convex} shows strict dissipativity for all $\ell_\mu$, $\mu\in[0,1)$, which by Theorem \ref{thm:discontinuous} implies continuity of $x^e_\mu$ on $[0,1-\eps]$ for all $\eps>0$ and thus on $[0,1)$. 
		\end{example}

		\subsection{A sufficient condition for small changes in $\mu$}
		
		In the previous section we have seen that the optimal equilibria depend continuously on $\mu$ in case strict dissipativity holds. In this section we will see that strict dissipativity also depends---in a certain sense---continuously on $\mu$. More precisely, we will give conditions under which strict dissipativity in $\mus\in(0,1)$ implies strict dissipativity for small variations of $\mus$, i.e. for $\mu \in(\mus-\eps,\mus+\eps)$. 
		For this purpose, we assume that the constraint set $\Y$ is defined in terms of inequality constraints \eqref{eq:Yg}. Further, we observe that if the system 
		is strictly $(x,u)$-dissipative at some equilibrium $\equ$, then this equilibrium is the unique minimiser of the constrained optimisation problem \eqref{eq:opteq} and the unique minimiser of the rotated cost $\tell$. This is because from the strict $(x,u)$-dissipativity we can conclude that the rotated stage cost is bounded from below by $\tell(x,u)\geq \alpha(\norm{x-\eq,u-u^{e}})$ for all $(x,u)\in\Y$ and, thus
		\[0=\tell\equ =\min_{(x,u)\in\Y}\tell(x,u) < \tell(x,u)\]
		for all $(x,u)\in\Y$ with $(x,u)\ne \equ$. Further, we briefly recall some definitions from nonlinear programming, see for instance \cite[Chapter 3]{Bertsekas1999}. Namely, considering the constrained optimisation problem \eqref{eq:opteq} with $\Y$ of the form \eqref{eq:Yg} we call a point $(x^\star,u^\star)$ satisfying $h(x^\star,u^\star)=0$ \textit{regular} if $\nabla h_1(x^\star,u^\star),\dots, \nabla h_n(x^\star,u^\star)$ are linearly independent with $h(x,u):=x-f(x,u)$. For the second order sufficiency conditions we refer to \cite[Proposition 3.3.2]{Bertsekas1999} and for the \textit{strong} second order sufficiency conditions see, for instance, \cite{Jittorntrum1984, Robinson1980}.

		\begin{theorem}\label{thm:mu+eps}
			Assume that
			\begin{enumerate}[(i)]
				\item the functions $f,\ell_1,\ell_2,g$ are twice continuously differentiable and $\Y$ is bounded.
				\item the system \eqref{eq: system} 
					is strictly $(x,u)$-dissipative for the cost function $\ell_{\mus}= \mus\ell_1+(1-\mus)\ell_2$ at the equilibrium $(\eq_{\mus},u^{e}_{\mus})$ for some $\mus\in[0,1]$ and the corresponding storage function $\lambda_{\mus}$ is twice continuously differentiable.
				\item the equilibrium $(\eq_{\mus},u^{e}_{\mus})$ is a regular point of problem \eqref{eq:opteq} and satisfies the strong second order sufficiency conditions for problem \eqref{eq:opteq}, with $\ell=\ell_{\mus}$ and for $\min_{(x,u)\in\Y}\tell_{\mus}$ with $\Y$, respectively, defined as in \eqref{eq:Yg}.
			\end{enumerate}
			Then there exists $\eps>0$ such that for all $\mu\in(\mus-\eps,\mus+\eps)\cap [0,1]$, there exists an equilibrium $\equm$ such that the system \eqref{eq: system} 
			is strictly $(x,u)$-dissipative for stage cost $\ell_\mu=\mu\ell_1+(1-\mu)\ell_2$.
		\end{theorem}
		
		This theorem is similar to other theorems in the literature, such as \cite[Theorem 5]{Muller_et_al_TAC15}, where small changes in the constraints are considered and \cite[Theorem 8.2]{Gruene2021}, where small changes in the discount factor are considered. As explained in \cite[Remark 8]{Muller_et_al_TAC15}, the proof of Theorem \ref{thm:mu+eps} follows by a slight modification of the proof of \cite[Theorem 5]{Muller_et_al_TAC15}. Hence, we omit it here. 
		\begin{remark} Since the reasoning used in these proofs was already used for multi-dimen\-sio\-nal parameters in \cite[Theorem 5]{Muller_et_al_TAC15}, the result extends readily to more than two cost functions.\end{remark}

		\subsection{Sufficient conditions for all $\mu\in [0,1]$}
		
		In this section we give sufficient conditions under which we can ensure strict dissipativity for all $\mu\in[0,1]$ provided we have strict dissipativity for $\mu=0$ and $\mu=1$. We start with a theorem that shows that this is always true if the optimal equilibrium does not depend on $\mu$.
		
		\begin{theorem}
			Assume that the system \eqref{eq: system} 
			is strictly dissipative for the cost functions $\ell_1$ and $\ell_2$ at the same equilibrium $x^e$. Then the system 
			is strictly dissipative for the cost function $\ell_\mu = \mu \ell_1 + (1-\mu)\ell_2$ for all $\mu\in[0,1]$.
		\end{theorem}
		\begin{proof} For $\ell_1$ and $\ell_2$ there are storage functions $\lambda_1$, $\lambda_2$ as well as $\KK_\infty$ functions $\alpha_1$, $\alpha_2$ such that we have the inequalities
			\beaa
			\lambda_1(f(x,u))  & \le & \lambda_1(x) + \ell_1(x,u) - \ell_1(x^e,u^e) - \alpha_1(\|x-x^e\|) \\
			\lambda_2(f(x,u))  & \le & \lambda_2(x) + \ell_2(x,u) - \ell_2(x^e,u^e) - \alpha_2(\|x-x^e\|).
			\eeaa
			Adding $\mu$-times the first equation and $(1-\mu)$-times the second equation yields
		\begin{align*}
			 \mu\lambda_1(f(x,u)) &+ (1-\mu)\lambda_2(f(x,u)) \\
			 &  \le  \mu\lambda_1(x) +  (1-\mu)\lambda_2(x) + \mu\ell_1(x,u) + (1-\mu)\ell_2(x,u)\\
			&  - \; \mu\ell_1(x^e,u^e) - (1-\mu) \ell_2(x^e,u^e)\\
			& - \; \mu\alpha_1(\|x-x^e\|) - (1-\mu)\alpha_2(\|x-x^e\|) .
			\end{align*}
			Defining $\lambda_\mu = \mu\lambda_1 + (1-\mu)\lambda_2$ and $\alpha_\mu = \mu\alpha_1 + (1-\mu)\alpha_2$ one thus obtains
			\[ \lambda_\mu(f(x,u)) \le \lambda_\mu(x) + \ell_\mu(x,u) - \ell_\mu(x^e,u^e) - \alpha_\mu(\|x-x^e\|). \]
			Since one easily checks that $\lambda_\mu$ is bounded from below (since both $\lambda_1$ and $\lambda_2$ are) and that $\alpha_\mu\in \KK_\infty$ for all $\mu\in[0,1]$, this shows the desired inequality \eqref{eq:diss} for all these $\mu$.
		\end{proof}
		
		\begin{remark} The proof remains completely identical for more than two cost functions, hence the statement holds accordingly in this case.\end{remark}
		
		When $\equm$ depends on $\mu$, the situation becomes more complicated. The construction we present in the remainder of this section is motivated by the linear quadratic result from Theorem \ref{thm:lq}. This theorem shows that the storage function for $\ell_\mu$ can be obtained by adding a linear term to the convex combination of the storage functions for $\ell_1$ and $\ell_2$. It follows from Proposition \ref{prop:lambdaLagrange} that this linear correction must be such that the gradient of the resulting storage function at $\equm$ equals the Lagrange multiplier of the optimal equilibrium problem 
		\begin{equation}\label{eq:optEqProb}
			\begin{split}
				\min_{(x,u)\in\Y}\mu\ell_1(x,u)&+(1-\mu)\ell_2(x,u)\\
				\mbox{s.t. } x&=f(x,u),
			\end{split}
		\end{equation}
		which we denote by $\nu_\mu$ and which satisfies the necessary optimality conditions \eqref{eq:Lagrange} for $\ell=\ell_\mu$. This idea was used before in \cite[Theorem 5]{Muller_et_al_TAC15}. We provide the following theorem in two versions. Theorem \ref{thm:linearcorrection} yields only local strict dissipativity (cf.\ Part (iv) of Definition \ref{def: strDiss}), while Theorem \ref{thm:linearcorrection2} yields standard (i.e., not only local) strict dissipativity, yet under stronger assumptions.
		
		\begin{theorem}\label{thm:linearcorrection}
			Assume that the system \eqref{eq: system} 
			is strictly dissipative for cost functions $\ell_1$ and $\ell_2$. Suppose that $\ell_1$ and $\ell_2$ as well as the corresponding storage functions $\lambda_1$ and $\lambda_2$ are twice continuously differentiable. Furthermore, assume that for each $\mu\in[0,1]$ the optimal equilibrium satisfies $\equm\in\rint \Y$ and that there exist $m_2(\mu)>m_1(\mu)\geq 0$ such that 
			\begin{align}
				\nabla^2_{(x,u)}\Big(\mu\tell_1\equm+(1-\mu)\tell_2\equm\Big) \geq m_2(\mu)I \label{strong_convexity_gamma} 
			\end{align}
			for $\tilde \ell_i$ from \eqref{eq:diss} and
			\begin{align}
				\nabla^2_{(x,u)}\Big(\tilde\lambda_{\mu}^Tf\equm\Big) \leq m_1(\mu)I \label{condition_nabla^2_f(x,u)}
			\end{align}
			for all $\mu\in[0,1]$, where
			\begin{align}
				\tilde\lambda_{\mu} = \nu_{\mu} - \mu \nabla_x\lambda_1(\eqm) - (1-\mu)\nabla_x\lambda_2(\eqm) \in \mathbb{R}^n \label{def_tilde_lambda_mu}
			\end{align}
			and $\nu_\mu$ is the Lagrange multiplicator for \eqref{eq:optEqProb}.
			Then, the system 
			is locally strictly dissipative for cost function $\ell_{\mu}$ for all $\mu\in[0,1]$ with storage function
			\begin{align}
				\lambda_{\mu}(x) = \mu\lambda_1(x) + (1-\mu)\lambda_2(x) + \tilde\lambda_{\mu}^Tx.  \label{def_lamba_mu}
			\end{align}
		\end{theorem}
		
		\begin{proof}Define
			\begin{align}
				\tell_{\mu}(x,u):=\ell_{\mu}(x,u)-\ell_{\mu}\equm + \lambda_{\mu}(x) - \lambda_{\mu}(f(x,u)) \label{def_gamma_mu}
			\end{align}
			with $\lambda_{\mu}$ from~\eqref{def_lamba_mu}. In the following, we show that $\equm$ is a strict local minimiser of $\tell_{\mu}$, implying local strict dissipativity of the system 
			for cost function $\ell_{\mu}$. This will be done by showing that $\nabla_{(x,u)}\tell_{\mu}\equm=0$ and $\nabla^2_{(x,u)}\tell_{\mu}\equm>0$.
			
			First, we define $h(x,u)=x-f(x,u)$ and $\Lambda_{\mu}(x,u)=\lambda_{\mu}(x) - \lambda_{\mu}(f(x,u))$ and note that
			\begin{align*}
				\nabla_{(x,u)}\Lambda_{\mu}\equm&=\nabla_{(x,u)}\lambda_{\mu}(\eqm)-\nabla_{(x,u)}\lambda_{\mu}(f\equm)\\
				&=[I \quad0 ]^T\nabla_x\lambda_{\mu}\eqm-\nabla_{(x,u)}f\equm\nabla_x\lambda_{\mu}(f\equm)\\
				&=\left([I\quad 0]^T - \nabla_{(x,u)}f(x,u)\right)\nabla_x\lambda_{\mu}(\eqm) \\
				&= \nabla_{(x,u)}h\equm\nabla_x \lambda_{\mu}(\eqm),
			\end{align*}
			where the second equation follows from the chain rule and the third follows from the equilibrium property of $\equm$. Further, by using the definitions of $\lambda_{\mu}$ and $\tilde\lambda_{\mu}$ in equations~\eqref{def_lamba_mu} and~\eqref{def_tilde_lambda_mu}, respectively, we obtain the derivative of the storage function
				\begin{align*}
				\nabla_x\lambda_{\mu}(\eqm)&=\nabla_x(\mu\lambda_1(\eqm)+(1-\mu)\lambda_2(\eqm)+\tilde{\lambda}_{\mu}^T\eqm)\\
				&=\mu\nabla_x\lambda_1(\eqm)+(1-\mu)\nabla_x\lambda_2(\eqm)+\tilde{\lambda}_{\mu}\\
				&=\mu\nabla_x\lambda_1(\eqm)+(1-\mu)\nabla_x\lambda_2(\eqm)+\nu_{\mu}-\mu\nabla_x\lambda_1(\eqm)\\
				&\phantom{\mu\nabla_x\lambda_1(\eqm)+(1-\mu)\nabla_x\lambda_2(\eqm)+\nu_{\mu}}-(1-\mu)\nabla_x\lambda_2(\eqm)\\
				&=\nu_{\mu}
			\end{align*}
			Using the two equations above we get that
			\begin{align*}
				\nabla_{(x,u)}\tell_{\mu}\equm &= \nabla_{(x,u)}\ell_{\mu}\equm + \nabla_{(x,u)}\Lambda_{\mu}\equm \\
				&= \nabla_{(x,u)}\ell_{\mu}\equm + \nabla_{(x,u)}h\equm^T\nabla_x\lambda_{\mu}(\eqm) \\
				&= \nabla_{(x,u)}\ell_{\mu}\equm + \nabla_{(x,u)}h\equm^T\nu_{\mu} \\
				&=0.
			\end{align*}
			Here, the last equality follows since it corresponds to the KKT conditions of problem \eqref{eq:optEqProb}, similar as in the proof of Proposition \ref{prop:lambdaLagrange}.
			
			
			Furthermore, we obtain
			\begin{align*}
				&	\nabla^2_{(x,u)}\tell_{\mu}\equm \\
				&= \nabla_{(x,u)}^2\ell_{\mu}\equm+\nabla_{(x,u)}^2\Lambda_{\mu}\equm\\
				&=\nabla_{(x,u)}^2\ell_{\mu}\equm+\nabla_{(x,u)}^2\left(\mu\lambda_1(\eqm)+(1-\mu)\lambda_2(\eqm)+\tilde{\lambda}_{\mu}^T\eqm\right)\\
				& \qquad \quad \qquad-\nabla_{(x,u)}^2\left(\mu\lambda_1(f\equm)+(1-\mu)\lambda_2(f\equm)+\tilde{\lambda}_{\mu}^Tf\equm\right)\\
				&=\nabla_{(x,u)}^2\left(\mu\tell_1\equm+(1-\mu)\tell_2\equm\right)+\nabla_{(x,u)}^2 \tilde{\lambda}_{\mu}^T\eqm - \nabla_{(x,u)}^2\tilde{\lambda}_{\mu}^Tf\equm\\
				&=\nabla_{(x,u)}^2\left(\mu\tell_1\equm+(1-\mu)\tell_2\equm\right)- \nabla_{(x,u)}^2\tilde{\lambda}_{\mu}^Tf\equm\\
				&\geq (m_2(\mu)-m_1(\mu))I > 0,
			\end{align*}
			which finishes the proof of the theorem. \end{proof}
		\begin{remark}
			The assumption $\equm\in \rint\Y$ in Theorem \ref{thm:linearcorrection} can be omitted if we assume that the second order sufficiency conditions hold for problem \eqref{eq:optEqProb}. Then, we have to use similar assumptions as in Theorem \ref{thm:mu+eps} and the corresponding proof of \cite[Theorem~5]{Muller_et_al_TAC15}.
		\end{remark}
		
		\begin{remark} The assertion and proof of Theorem \ref{thm:linearcorrection} can be extended to $m$ cost functions $\ell_1,\ldots,\ell_m$ if all convex combinations with two functions ($\mu \tilde\ell_1 + (1-\mu)\tilde \ell_2$, $\mu \nabla_x\lambda_1+(1-\mu)\nabla_x\lambda_2$, $\mu \lambda_1+(1-\mu)\lambda_2$, etc.) in the assumptions and in the proof are replaced by the respective convex combinations of $m$ cost functions $\sum_i\mu_i\tilde \ell_i$, $\sum_i\mu_i\nabla_x\lambda_i$, etc.\label{rem:ext_linearcorrection}\end{remark}
		
		\begin{remark}
			It follows from Proposition \ref{prop:lambdaLagrange} that if a storage function of the form \emph{convex combination of $\lambda_1$ and $\lambda_2$ plus linear correction} exists for a given value of $\mu$, then it must be of the form \eqref{def_lamba_mu}. However, as the following example shows, this construction may fail to produce a valid storage function.
		\end{remark}
		
		\begin{example} \label{ex:linearfail} Consider
			\[ x^+ = f(x,u) = 2x-x^2+u+u^2+u^3 \]
			with cost functions 
			\[ \ell_1(x,u) = 2x^2 + 0.0001 u^2 \quad \mbox{ and } \quad \ell_2(x,u) = 2x^2 + 0.9999 u^2 + 2 u.\]
			For the purpose of illustrating the connection to multiobjective optimal control, we consider an optimal control problem of the form \eqref{eq:ocp} with stage cost $\ell_\mu = \mu\ell_1 +(1-\mu)\ell_2$, $N=10$, and $x_0=1$. We can solve this problem numerically by choosing a finite set of weights from the interval $[0,1]$, for more details see \cite{Ehrgott2005}. By plotting the two objectives against each other for the resulting minimisers (where $\ell_1$, $\ell_2$ correspond to $J_1$, $J_2$, respectively) we get the nondominated set (also known as Pareto-front), which is illustrated in Figure \ref{fig:pareto}.
			\begin{figure}[h]
				\begin{center}
					\includegraphics[width=0.5\textwidth]{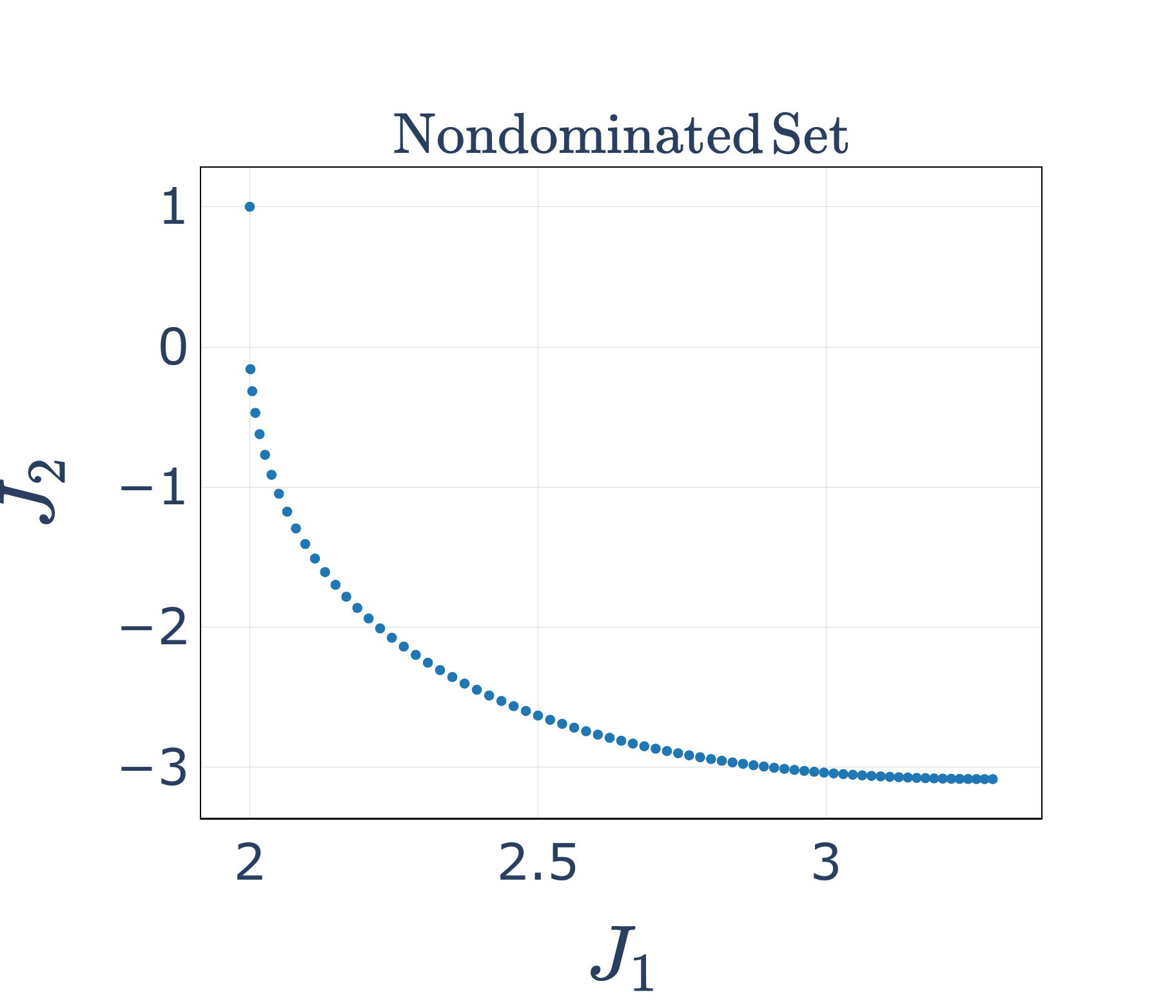}
					\caption{Nondominated set of the corresponding optimal control problem \label{fig:pareto}}
				\end{center}
			\end{figure}
		We like to remark here that the weighted sum approach does not parametrise all of the efficient solution, as already mentioned in the introduction. This may also be reason for the small gap on the left side of the nondominated set.
			
			For both cost functions, strict dissipativity holds locally with storage functions and optimal equilibrium
			\[ \lambda_1(x) = 0, \quad (x^e_1,u^e_1)=(0,0) \]
			for $\ell_1$ and 
			\[ \lambda_2(x)= 2.1986096 x, \quad (x^e_2,u^e_2)= (0.2618259, -0.2357480)\]
			for $\ell_2$ (all values are rounded to 6 or 7 digits). For $\ell_1$ this is obvious since the function is strictly convex, hence strict dissipativity holds even globally, while for $\ell_2$ one checks that the rotated cost $\tilde\ell_2$ has a positive definite second derivative in $(x^e_2,u^e_2)$.
			
			For $\mu=0.5$ one computes the Lagrange multiplier of \eqref{eq:optEqProb} as $\nu_{0.5}=1.111667$. The corresponding optimal equilibrium is \\
			$(x^e_{0.5},u^e_{0.5}) = (0.1786289, -0.1709482)$. This means that if a linear storage function $\lambda_{0.5}$ exists, then it must be of the form $\lambda_{0.5} =  1.111667 x$, and this is the storage function constructed in equation \eqref{def_lamba_mu} given that $\lambda_1$ and $\lambda_2$ are linear. For the rotated cost
			\[ \tilde\ell_{0.5}(x,u) = 2x^2 + 0.5u^2 + u + \nu_{0.5}x - \nu_{0.5}f(x,u) \]
			one computes that $\partial^2/\partial u^2 \tilde\ell_{0.5}(x^e_{0.5},u^e_{0.5}) = -0.306538$, implying that $\tilde\ell_{0.5}$ is not convex in $(x^e_{0.5},u^e_{0.5})$, which would be a necessary condition for local strict dissipativity.
		\end{example}
		
		We now proceed to a non-local version of Theorem \ref{thm:linearcorrection}, which is achieved by extending the convexity assumptions from Theorem \ref{thm:linearcorrection} to all $(x,u)\in\Y$.
		
		\begin{theorem}\label{thm:linearcorrection2}
			Assume that the system \eqref{eq: system} 
				is strictly dissipative for cost functions $\ell_1$ and $\ell_2$. Suppose that $\ell_1$ and $\ell_2$ as well as the corresponding storage functions $\lambda_1$ and $\lambda_2$ are twice continuously differentiable. Furthermore, assume that for each $\mu\in[0,1]$ the optimal equilibrium satisfies $\equm\in\rint \Y$ and that there exist $m_2(x,u,\mu)>m_1(x,u,\mu)\geq 0$ such that 
			\begin{align}
				\nabla^2_{(x,u)}\Big(\mu\tell_1(x,u)+(1-\mu)\tell_2(x,u)\Big) \geq m_2(x,u,\mu)I \label{eq: strong_convexity_tell} 
			\end{align}
			and
			\begin{align}
				\nabla^2_{(x,u)}\Big(\tilde\lambda_{\mu}^Tf(x,u)\Big) \leq m_1(x,u,\mu)I \label{eq:condition_nabla^2_f(x,u)}
			\end{align}
			for all $(x,u)\in\Y$, $\mu\in[0,1]$, where
			\begin{align}
				\tilde\lambda_{\mu} = \nu_{\mu} - \mu \nabla_x\lambda_1(\eqm) - (1-\mu)\nabla_x\lambda_2(\eqm) \in \mathbb{R}^n \label{eq:def_tilde_lambda_mu}
			\end{align}
			and $\nu_\mu$ is the Lagrange multiplicator for \eqref{eq:optEqProb}.
			Then, the system 
			is strictly dissipative for cost function $\ell_{\mu}$ for all $\mu\in[0,1]$ with storage function
			\begin{align}
				\lambda_{\mu}(x) = \mu\lambda_1(x) + (1-\mu)\lambda_2(x) + \tilde\lambda_{\mu}^Tx.  \label{eq:def_lamba_mu}
			\end{align}
		\end{theorem}
		
		\begin{proof}
			
			To show that $\equm$ is a minimiser we proceed exactly as in the proof of Theorem \ref{thm:linearcorrection}.
			In order to conclude that $\equm$ is the strict minimiser we show that the function $\tell_{\mu}$ is strictly convex by considering the second derivative for all $(x,u)\in\Y$. With the same computation as in the proof of Theorem \ref{thm:linearcorrection} for the second derivative but now for all $(x,u)$ instead of only for $\equm$, we obtain
			\begin{align*}
				\nabla^2_{(x,u)}\tell_{\mu}(x,u) &
				=\nabla_{(x,u)}^2\left(\mu\tell_1(x,u)+(1-\mu)\tell_2(x,u)\right)- \nabla_{(x,u)}^2\tilde{\lambda}_{\mu}^Tf(x,u)\\
				&\geq (m_2(x,u,\mu)-m_1(x,u,\mu))I > 0,
			\end{align*}%
			for all $(x,u)\in\Y$ and $\mu\in[0,1]$. Hence, $\tell_{\mu}$ is strictly convex and continuous and together with $\tell_{\mu}\equm=0$ positive definite. Thus, we can conclude, that $\tell_{\mu}$ is radially unbounded, see \cite{Baker2016}, and therefore, see \cite[Lemma 4.3]{Khalil2001}, that there exists a $\KK_\infty-$function $\alpha$ such that $\tell_{\mu}(x,u)\geq \alpha(\norm{x-\eqm})$. 
		\end{proof}
		
		We note that Remark \ref{rem:ext_linearcorrection} applies accordingly to Theorem \ref{thm:linearcorrection2}.
		
		\begin{remark}
			We can relax inequality \eqref{eq: strong_convexity_tell} by assuming that both $\tell_1$ and $\tell_2$ are lower bounded by functions $\bar\ell_1$ and $\bar\ell_2$ such that $\tell_i\equm = \bar\ell_i\equm$ for all $\mu\in[0,1]$. Then inequality \eqref{eq: strong_convexity_tell} only needs to hold with $\bar\ell_i$ in place of $\tell_i$, i.e., we require that $\nabla_{(x,u)}\bar\ell_{\mu}\equm = \nabla_{(x,u)}\tilde\ell_{\mu}\equm = 0$ and $\nabla^2_{(x,u)}\bar\ell\equm > 0$. Together with the fact that $\bar\ell_{\mu}$ is a lower bound for $\tilde\ell_{\mu}$ and $\tell_i\equm = \bar\ell_i\equm$ for all $\mu\in[0,1]$, this yields the desired result.
		\end{remark}
		
		In the following, we illustrate Theorem \ref{thm:linearcorrection2} with an economic example originally introduced in \cite{Brock1973} and adapted to our setting.
		\begin{example}
			Consider the nonlinear system 
			\[x^+ =x^3-2 x^2 + u,\]
			and the two convex economic stage costs 
			\begin{align*}
				\ell_1(x,u)&= \ln(5x^{0.34}-u),\\
				\ell_2(x,u)&=\ln(3x^{0.2}-u),
			\end{align*}
			which were originally given in \cite{Brock1973}. Further, we impose state and control constraint sets $\X=[0,10]$ and $\U=[0.1,5]$. For each stage cost $\ell_i$, we determine the optimal equilibrium each given by 
			\[(x^{e}_1,u^{e}_1)=(0.6214, 1.1537)\quad \text{and} \quad (x^{e}_2,u^{e}_2)=(0.2507,0.3607).\]
			The system is strictly dissipative for stage costs $\ell_i$ at their corresponding equilibria $\equi$ with storage functions
			\[\lambda_1(x)=0.3226 x\quad \text{and}\quad \lambda_2(x)= 0.5223x\]
			and rotated stage costs
			\begin{align*}
				\tell_1(x,u)&= \ln(5x^{0.34}-u)+1.1312 +0.3226 x -0.3226(x^3-2x^2+u),\\
				\tell_2(x,u)&=\ln(3x^{0.2}-u) +0.6493 +0.5223x -0.5223(x^3-2x^2+u).
			\end{align*}
			It is easy to see that all occurring functions are twice continuously differentiable and, thus, we can check numerically the first conditions of Theorem \ref{thm:linearcorrection2}. For each $\mu\in[0,1]$ the optimal equilibrium lies in the interior of $\Y$, i.e. $\equm\in\rint (\X\times\U)$. Moreover, we can set $m_2(x,u,\mu)=5.9$ by using the minimum of all calculated second derivatives of $\mu\tell_1+(1-\mu)\tell_2$ with $\mu\in[0,1]$. Next, we have to check the condition on the correction term. To this end, we calculate the Lagrange multiplier of problem \eqref{eq:optEqProb} as well as the first derivatives of the storage functions $\nabla_x \lambda_1(\eqm)$ and $\nabla_x \lambda_2(\eqm)$. Doing so, we can estimate $m_1(x,u,\mu)=-2.1699\cdot 10^{-11}\approx0$. We can therefore conclude, that the system is strictly dissipative for stage cost $\mu\ell_1 +(1-\mu)\ell_2$ at the corresponding equilibrium $\equm$ for all $\mu\in[0,1]$ with storage function 
			\[\lambda_\mu(x)=\mu \lambda_1(x)+(1-\mu)\lambda_2(x)+ \tilde\lambda_\mu^T x,\]
			with $\tilde\lambda_\mu\in (0, 0.0144)$.
		\end{example}
		
		\section{Conclusion}\label{sec:con}
		
		For optimal control problems we have investigated strict dissipativity for stage costs given by convex combinations of cost functions for which strict dissipativity holds. For linear quadratic problems, strict dissipativity for the weighted stage costs follows under mild regularity conditions. For nonlinear problems, general statements require quite restrictive assumptions, such as independence of the optimal equilibrium of the weight parameters, while less restrictive assumptions turned out to be quite technical. Nevertheless, our paper provides a set of tools that could be valuable and useful particularly in the context of multiobjective MPC.


		\bibliographystyle{siam_nolines}
		\bibliography{dissipativity_MO}
		
	\end{document}